\documentclass[11pt]{amsart}
\usepackage{amsmath}
\usepackage{amssymb}
\usepackage{amscd}

\def\NZQ{\mathbb}               
\def\NN{{\NZQ N}}
\def\QQ{{\NZQ Q}}
\def\ZZ{{\NZQ Z}}

\def\CC{{\NZQ C}}

\newtheorem{Theorem}{Theorem}[section]
\newtheorem{Lemma}[Theorem]{Lemma}
\newtheorem{Corollary}[Theorem]{Corollary}

\newtheorem{Example}[Theorem]{Example}

\let\epsilon\varepsilon
\let\phi=\varphi
\let\kappa=\varkappa

\def \a {\alpha}
\def \s {\sigma}
\def \d {\delta}
\def \g {\gamma}

\begin{document}

\title{Algebraic series and valuation rings\\ over nonclosed fields}

\author{Steven Dale Cutkosky  and Olga Kashcheyeva}
\thanks{The first author was partially supported by NSF }


\maketitle

\section{Introduction}

Suppose that $k$ is an arbitrary field. Consider the field
$k((x_1,\ldots,x_n))$, which is the quotient field of the  ring
$k[[x_1,\ldots,x_n]]$ of formal power series  in the variables
$x_1,\ldots,x_n$, with coefficients in $k$. Suppose that
$\overline k$ is an algebraic closure of $k$, and $\sigma\in
\overline k[[x_1,\ldots,x_n]]$ is a formal power series. In this
paper, we give a very simple necessary and sufficient condition
for $\sigma$ to be algebraic over $k((x_1,\ldots,x_n))$. We  prove the following theorem, which is restated in an equivalent formulation in Theorem  \ref{multi}.

\begin{Theorem}\label{varfin} Suppose that $k$ is a field of characteristic $p\ge0$,
with algebraic closure $\overline k$.  Suppose that
$$
\sigma(x_1,\ldots,x_n)=\sum_{i_1,\ldots,i_n\in\NN}\alpha_{i_1,\ldots,i_n}
x_1^{i_1}x_2^{i_2}\cdots x_n^{i_n}\in \overline k[[x_1,\ldots,x_n]]
$$
where  $\alpha_{i_1,\ldots,i_n}\in \overline k$ for all $i$. Let
$$
L= k(\{\alpha_{i_1,\ldots,i_n}\mid i_1,\ldots,i_n\in\NN\})
$$
 be the extension field of $k$
generated by the coefficients of $\sigma(x_1,\ldots,x_n)$. Then
$\sigma(x_1,\ldots,x_n)$ is algebraic over $k((x_1,\ldots,x_n))$
if and only if there exists $r\in\NN$ such that
$[kL^{p^r}:k]<\infty$, where $kL^{p^r}$ is the compositum of $k$
and $L^{p^r}$ in $\overline k$.

\end{Theorem}

In the case that $L$ is separable over $k$ (Corollary \ref{2}), or  that $k$ is a finitely
generated extension field of a perfect field (Corollary \ref{1}), we have a stronger condition. In these cases, $\sigma$ is algebraic
if and only if  $[L:k]<\infty$. The finiteness condition $[L:k]<\infty$ does not characterize algebraic series over arbitrary base fields $k$ of positive characteristic. To illustrate this, we give a simple example, in Example \ref{example}, of an algebraic series in one variable
for which $[L:k]=\infty$.

 In Section 2, we prove Theorem \ref{varfin}
in the case $n=1$. The most difficult part of the proof arises when $k$ is not perfect. Our proof uses the theorem of resolution of singularities of
a germ of a plane curve singularity over an arbitrary field (c.f. \cite{Ab1}, \cite{O}, \cite{C}). In Section 3, we prove Theorem \ref{varfin} for
any number of variables $n$. The proof  involves induction on the number of variables, and uses the result for one variable proven in Section 2.

In the case when $k$ has characteristic zero and $n=1$, the conclusions
of Theorem \ref{varfin} are classical. We recall the very strong known results,
under the assumption that $k$ has characteristic zero, and there is only one variable ($n=1$).
The algebraic
closure of the field of formal meromorphic power series $k((x))$ in the variable $x$  is
\begin{equation}\label{eqI1}
\overline{k((x))}=\cup_F\cup_{n=1}^{\infty}F((x^{\frac{1}{n}}))
\end{equation}
where $F$ is any finite field extension of $k$ contained in the algebraic
closure $\overline k$ of $k$.  The equality (\ref{eqI1}) is stated and proven  in Ribenboim and Van den Dries' article \cite{RVdD}.
A proof can also be deduced from Abhyankar's Theorem (\cite{Ab3} or Section 2.3 of \cite{GM}).
The equality (\ref{eqI1}) already follows for an algebraically closed field $k$ of characteristic zero
from a classical algorithm of Newton \cite{BK}, \cite{C}.

If $k$ has characteristic $p>0$, then the algebraic closure of
$k((x))$ is much more complicated, even when $k$ is algebraically closed,
because of the existence of Artin Schreier extensions, as  is shown in
Chevalley's book \cite{Ch}. In fact,  the series
\begin{equation}\label{eqC}
\sigma(x)=\sum_{i=1}^{\infty}x^{1-\frac{1}{p^i}},
\end{equation}
considered by Abhyankar in \cite{Ab4},
is algebraic over $k((x))$, as it satisfies the relation
$$
\sigma^p-x^{p-1}\sigma-x^{p-1}=0.
$$
When $k$ is an algebraically closed field of arbitrary characteristic, the
 ``generalized power series'' field $k((x^{\QQ}))$ is algebraically closed, as is shown by Ribenboim
in  \cite{R}. The approach of studying the algebraic closure of $k((x))$
through generalized power series is developed by Benhessi \cite{Be},
Hahn \cite{Ha}, Huang, \cite{Hu}, Poonen \cite{Po}, Rayner \cite{Ra}, Stefanescu \cite{St} and Vaidya \cite{V}. A complete solution when $k$ is a perfect field
is given by Kedlaya in \cite{Ke}. He shows that the algebraic closure of $k((x))$ consists of all ``twist recurrent series'' $u=\sum\alpha_ix^i$
in $\overline k((x^{\QQ}))$
such that all $\alpha_i$ lie in a common finite extension of $k$.

When $n>1$, the algebraic closure of $k((x_1,\ldots,x_n))$ is known to be extremely complicated, even when $k$ is algebraically closed of characteristic
0. In this case,  difficulties occur when the ramification locus of
a finite extension is very singular. There is a good understanding in some
important cases, such as when the ramification locus is a simple normal crossings divisor and the characteristic of $k$ is 0 or the ramification is
tame (Abhyankar \cite{Ab3}, Grothendieck and Murre \cite{GM}) and for quasi-ordinary singularities (Lipman \cite{L2}, Gonz\'alez-P\'erez \cite{G-P}).

More generally, subrings of a power series ring can be very
complex, and are a source of many extraordinary examples, such as
\cite{N}, \cite{Ro}, \cite{HRS}.

As an application of our methods, we give a characterization of valuation rings $V$ which dominate
an excellent, Noetherian local domain $R$ of dimension two, and such that  the rank increases after
passing to the completion of a birational extension of $R$. The characterization is   known
when  the residue field of $R$ is algebraically closed (Spivakovsky \cite{S}).
In this case ($R/m_R$ algebraically closed) the rank increases under completion if and only if $\mbox{dim}_R(V)=0$ ($V/m_V$ is
algebraic over $R/m_R$) and $V$ is discrete of rank 1.

However, the characterization is more subtle over non closed fields.
In Theorem \ref{TheoremV2}, we show that the condition that the rank  increases
under completion is characterized by the two conditions that the residue field of $V$ is finite  over the residue field of $R$, and that $V$ is  discrete of rank 1. The case when the residue field of $V$ is infinite algebraic over the residue field of  $R$ and the value group is discrete of rank 1 can occur,
and the rank of such a valuation does not increase when passing to completion.
In Corollary \ref{Cor}, we show that there exists a valuation ring $V$
dominating $R$ whose value group is discrete of rank 1 with $\mbox{dim}_R(V)=0$ such that
the rank of $V$ does not increase under completion if and only if the algebraic closure
of $R/m_R$ has infinite degree over $R/m_R$.

We point out the contrast of the  conclusions of
Theorem \ref{varfin} with the results of Section 4.  The finiteness condition
$[L:k]<\infty$ of the coefficient field of a series over a base field $k$
does not characterize algebraicity of a series in positive characteristic,
while the corresponding finiteness condition on residue field extensions does
characterize algebraicity (the increase of rank) in the case of valuations dominating a local ring of Theorem \ref{TheoremV2}. We illustrate this distinction in Example \ref{example2} by constructing the valuation ring determined by the series of
Example \ref{example}. We conclude by showing a simple standard power series representation of the valuation associated to the algebraic series of (\ref{eqC}), whose
exponents do not have bounded denominators.

The concept of the rank increasing when passing to the completion already
appears implicitely in Zariski's paper \cite{Z}. Some papers where the concept is developed
are \cite{S}, \cite{HS} and \cite{CG}.

If $R$ is a local (or quasi local) ring, we will denote its maximal ideal by $m_R$.

\section{Series in one variable}

\begin{Lemma}\label{Lemma1}
Suppose that $R$ is a 2 dimensional regular local ring, and $x\in m_R$ is part of a regular system of parameters.

Suppose that $k_0$ is a coefficient field of $\hat R$ and $y\in \hat R$ is such that $x,y$ are regular parameters in $\hat R$. This determinines an isomorphism
$$
\hat R\stackrel{\lambda_0}{\rightarrow} k_0[[x,y]]
$$
of $\hat R$ with a power series ring.
Suppose that $\alpha$ is separably algebraic over $k_0$. Let $y_1=\frac{y}{x}-\alpha$.
Then there exists a maximal ideal $n\subset R[\frac{m_R}{x}]$ and an isomorphism
$$
\widehat{R[\frac{m_R}{x}]_n}\stackrel{\lambda_1}{\rightarrow} k_0(\alpha)[[x,y_1]]
$$
which makes the diagram
$$
\begin{array}{ccc}
\hat R&\stackrel{\lambda_0}{\rightarrow}&k_0[[x,y]]\\
\downarrow&&\downarrow\\
\widehat{R[\frac{m_R}{x}]_n}&\stackrel{\lambda_1}{\rightarrow} &k_0(\alpha)[[x,y_1]]
\end{array}
$$
commute, where the vertical arrows are the natural maps.
\end{Lemma}

\begin{proof} There exists $\tilde y\in R$ such that $\tilde y
=y+h$ where $h\in m_R^3\hat R$. We have
$$
\frac{\tilde y}{x}-\alpha=y_1+\frac{h(x,x(y_1+\alpha))}{x}
=y_1+h_1(x_1,y_1)
$$
where $h_1\in k_0(\alpha)[[x,y_1]]$ is a series of order $\ge 2$.
Thus we have natural change of variables
$k_0[[x,y]]=k_0[[x,\tilde y]]$ and $k_0(\alpha)[[x,y_1]]=k_0(\alpha)[[x,\frac{\tilde y}{x}-\alpha]]$.
We may thus assume that $y\in R$.

We have a natural inclusion induced by $\lambda_0$,
$$
R[\frac{y}{x}]\subset \hat R[\frac{y}{x}]\subset k_0(\alpha)[[x,y_1]].
$$
Let $n=(x,y_1)\cap R[\frac{y}{x}]$.

Let $h(t)$ be the minimal polynomial of $\alpha$ over $k_0$, and $f\in R[\frac{y}{x}]$ be a lift of
$$
h(\frac{y}{x})\in k_0[\frac{y}{x}]\cong R[\frac{y}{x}]/xR[\frac{y}{x}].
$$
Then $n=(x,f)$ and we see that $R[\frac{y}{x}]/n\cong k_0(\alpha)$.
Now the conclusions of the lemma follow from Hensel's Lemma (c.f. Lemma 3.5 \cite{C}).

\end{proof}

\begin{Theorem}\label{main} Suppose that $k$ is a  field, with algebraic closure
$\overline k$. Let $\overline k((x))$ be the field  of formal Laurent series
in a variable $x$  with coefficients in $\overline k$. Suppose that
$$
\sigma(x)=\sum_{i=d}^{\infty}\alpha_ix^i\in \overline k((x))
$$
where  $d\in \ZZ$ and $\alpha_i\in \overline k$ for all $i$. Let $L=k(\{\alpha_i\mid i\in \NN\})$,
and suppose that $L$ is separable over $k$.  Then $\sigma(x)$ is algebraic over $k((x))$ if and
only if
$$
[L:k]<\infty.
$$
\end{Theorem}

\begin{proof} We reduce to the case where $d\ge 1$, by observing that
$\sigma$ is algebraic over $k((x))$ if and only if $x^{1-d}\sigma$ is.

First suppose that $[L:k]<\infty$. Let $M$ be a finite Galois extension of $k$ which contains $L$.
Let $G$ be the Galois group of $M$ over $k$. $G$ acts naturally by $k$ algebra isomorphisms on
$M[[x]]$, and the invariant ring of the action is $k[[x]]$. Let $f(y)=\prod_{\tau\in
G}(y-\tau(\sigma))\in M[[x]][y]$. Since $f$ is invariant under the action of $G$, $f(y)\in
k[[x]][y]$. Since $f(\sigma)=0$, we have that $\sigma$ is algebraic over $k((x))$.

Now suppose that $\sigma(x)=\sum_{i=1}^\infty\alpha_ix^i$ is algebraic over $k((x))$.  Then there
exists
$$
g(x,y)=a_0(x)y^n+a_1(x)y^{n-1}+\cdots +a_n(x)\in k[[x]][y]
$$
such that $a_0(x)\ne 0$, $n\ge 1$, $g$ is irreducible and $g(x,\sigma(x))=0$.

Let
$$
y_0=y, y_1=\frac{y}{x}-\alpha_1,y_2=\frac{y_1}{x}-\alpha_2,\ldots,
y_i=\frac{y_{i-1}}{x}-\alpha_{i},\cdots
$$
and define
$$
S_0=k[[x,y]], S_1=k(\alpha_1)[[x,y_1]], \cdots S_i=k(\alpha_1,\ldots,\alpha_i)[[x,y_i]],\ldots
$$
We have natural inclusions
$$
S_0\rightarrow S_1\rightarrow\cdots \rightarrow S_i\rightarrow \cdots
$$
By Lemma \ref{Lemma1}, there exists a sequence of inclusions
\begin{equation}\label{eq2}
R_0\rightarrow R_1\rightarrow\cdots\rightarrow R_i\rightarrow \cdots
\end{equation}
where $R_0=k[[x]][y]_{(x,y)}$ and each $R_i$ is a localization at a maximal ideal of the
blow up of the maximal ideal $m_{R_{i-1}}$ of $R_{i-1}$,
and we have a commutative diagram of homomorphisms
$$
\begin{array}{lllllllll}
S_0&\rightarrow &S_1&\rightarrow&\cdots&\rightarrow&S_i&\rightarrow&\cdots\\
\uparrow&&\uparrow&&&&\uparrow&&\\
R_0&\rightarrow &R_1&\rightarrow&\cdots&\rightarrow&R_i&\rightarrow&\cdots
\end{array}
$$
where the vertical arrows induce isomorphisms of the $m_{R_i}$-adic completions
$\hat R_i$ of $R_i$ with $S_i$.
We further have that $x$ is part of a
regular system of parameters in $R_i$ for all $i$, and $m_{R_{i-1}}R_i=xR_i$ for all $i$.

By our construction, we have that
\begin{equation}\label{eq3}
R_i/m_{R_i}\cong k(\alpha_1,\ldots,\alpha_{i})
\end{equation}
for all $i$.

For all $i$, write $g=x^{b_i}g_i$ where $g_i\in R_i$ and $x$ does not divide $g_i$ in $R_i$.

In $\overline k[[x,y_i]]$, we have a factorization
$$
y-\sigma=x^i(y_i-\sum_{j={i+1}}^{\infty}\alpha_jx^{j-i}).
$$
Since $y-\sigma$ divides $g$ in $\overline k[[x,y]]$, we have that
$y_i-\sum_{j={i+1}}^{\infty}\alpha_jx^{j-i}$ divides $g_i$ in
$\overline k[[x,y_i]]$. Thus $g_i$ is not a unit in $\overline k[[x,y_i]]$,
and is thus not a unit in $R_i$.

Let $C$ be the curve germ $g=0$ in the germ $\mbox{Spec}(R_0)$
of a nonsingular surface. The sequence (\ref{eq2}) is obtained by blowing up
the closed point in  $\mbox{Spec}(R_i)$, and localizing at a point
which is on the strict transform of $C$. $g_i=0$ is a local equation of
the strict transform of $C$ in $\mbox{Spec}(R_i)$. By embedded resolution
of plane curve singularities (\cite{Ab1}, \cite{O} or a simple generalization
of Theorem 3.15 and Exercise 3.13 of \cite{C}) we obtain that there exists $i_0$
such that the total transform of $C$ in $\mbox{Spec}(R_i)$ is a simple normal crossings divisor  for all $i\ge i_0$. Since $x^{b_i}g_i=g=0$ is a local
equation of the total transform of $C$ in $\mbox{Spec}(R_i)$, we have  that
$x,g_i$ are regular parameters in $R_i$ for all $i\ge i_0$.
Thus $g_{i_0}=x^{i-i_0}g_i$ for all $i\ge i_0$, and
$R_i=R_{i-1}[\frac{g_{i-1}}{x}]_{(x,g_i)}$ for all $i\ge i_0+1$.

We thus have that $R_i/m_{R_i}\cong R_{i_0}/m_{R_{i_0}}$ for all $i\ge i_0$,
and we see that
$$
L=\cup_{i\ge 0}R_i/m_{R_i}=R_{i_0}/m_{R_{i_0}} =k(\alpha_1,\ldots,\alpha_{i_0}).
$$
Thus $[L:k]<\infty$.

\end{proof}

\begin{Example}\label{example}  The conclusions of Theorem \ref{main} may fail if $L$ is not separable over $k$.
\end{Example}

\begin{proof} Let $p$ be a prime and $\{t_i\mid i\in\NN\}$ be algebraically independent over
the finite field $\ZZ_p$. Let $k=\ZZ_p(\{t_i\mid i\in\NN\})$. Define
$$
\sigma(x)=\sum_{i=1}^{\infty}t_i^{\frac{1}{p}}x^i\in \overline k[[x]].
$$
Let
$$
f(y)=y^p-\sum_{i=1}^{\infty}t_ix^{ip}\in k[[x]][y].
$$
$\sigma(x)$ is algebraic over $k[[x]]$ since
$$
f(\sigma(x))=(\sigma(x))^p-\sum_{i=1}^{\infty}t_ix^{ip}=0.
$$
However,
$$
[k(\{t_i^{\frac{1}{p}}\mid i\in\NN\}):k]=\infty.
$$
\end{proof}

Suppose that $k$ is a field of characteristic $p>0$ and $L$ is an extension
field of $k$. For $n\in \NN$, let
$$
L^{p^n}=\{f^{p^n}\mid f\in L\}.
$$
If $k$ has characteristic $p=0$, we take $L^{p^n}=L$ for all $n$.

\begin{Theorem}\label{insep} Suppose that $k$ is a field of characteristic $p>0$,
with algebraic closure
$\overline k$. Let $\overline k((x))$ be the field  of formal Laurent series
in the variable $x$ with coefficients in $\overline k$. Suppose that
$$
\sigma(x)=\sum_{i=d}^{\infty}\alpha_ix^i\in \overline k((x))
$$
where  $d\in \ZZ$ and $\alpha_i\in \overline k$ for all $i$.
Let $L= k(\{\alpha_i\mid i\in \NN\})$, and assume that $L$ is purely inseparable over $k$.           Then $\sigma(x)$ is algebraic over $k((x))$ if and only if
there exists $n\in\NN$ such that  $L^{p^n}\subset k$.

\end{Theorem}

\begin{proof} As in the proof of Theorem \ref{main}, we may assume that $d\ge 1$.

First suppose that  $L^{p^n}\subset k$ for some $n$. Then $\tau(x)=\sigma(x)^{p^n}\in k[[x]]$, and $\sigma(x)$ is the root of $y^{p^n}-\tau(x)=0$. Thus $\sigma$ is algebraic over $k((x))$.

Now suppose that
$\sigma(x)=\sum_{i=1}^\infty\alpha_ix^i\in \overline k[[x]]$ is algebraic over
$k((x))$.  Then there exists
$$
g(x,y)=a_0(x)y^n+a_1(x)y^{n-1}+\cdots +a_n(x)\in k[[x]][y]
$$
such that $a_0(x)\ne 0$, $n\ge 1$, $g$ is irreducible and $g(x,\sigma(x))=0$.

Let $K$ be the quotient field of $\overline k[[x]][y]$, and let $R_0:=S_0:=k[[x]][y]_{(x,y)}$. We
will first construct a series of subrings $S_i$ of $K$.

 Define a local $k$-algebra homomorphism $\pi_0:S_0\rightarrow \overline k[[x]]$
by prescribing that $\pi_0(x)=x$ and $\pi_0(y)=\sigma(x)$. The kernel of $\pi_0$ is the prime ideal $gS_0$.
$$
\frac{y}{x}=\sum_{i=0}^{\infty}\alpha_{i+1}x^i \in \overline k[[x]]
$$
defines a $k$-algebra homomorphism
$S_0[\frac{y}{x}]\rightarrow \overline k[[x]]$ which extends $\pi_0$. Let $\lambda(1)\in\NN$ be the
smallest natural number such that $\alpha_1^{p^{\lambda(1)}}\in k$. Then
the maximal ideal $x\overline k[[x]]$ of $\overline k[[x]]$ contracts to
$$
x\overline k[[x]]\cap S_0[\frac{y}{x}]=(x,\left(\frac{y}{x}\right)^{p^{\lambda(1)}}-\alpha_1^{p^{\lambda(1)}}).
$$
Set $y_1=\left(\frac{y}{x}\right)^{p^{\lambda(1)}}-\alpha_1^{p^{\lambda(1)}}$.
Let
$$
S_1=S_0[\frac{y}{x}]_{(x,y_1)}.
$$
Let $\pi_1:S_1\rightarrow \overline k[[x]]$ be the local $k$-algebra homomorphism
 induced by $\pi_0$.

We have that $x,y_1$ is a regular system of parameters in $S_1$, with
$$
y_1=\sum_{i=1}^{\infty}\alpha_{i+1}^{p^{\lambda(1)}}x^{ip^{\lambda(1)}}.
$$
$S_1/m_{S_1}\cong k(\alpha_1)$ and
$$
[S_1/m_{S_1}: S_0/m_{S_0}]=[k(\alpha_1):k]=p^{\lambda(1)}.
$$

Let $\lambda(2)\in\NN$ be the
smallest  natural number such that $\alpha_2^{p^{\lambda(1)+\lambda(2)}}\in k(\alpha_1)$.
Let
$$
y_2=\left(\frac{y_1}{x^{p^{\lambda(1)}}}\right)^{p^{\lambda(2)}}
-\alpha_2^{p^{\lambda(1)+\lambda(2)}}.
$$
Then there is an expansion in $\overline k[[x]]$
$$
y_2=\sum_{i=1}^{\infty}\alpha_{i+2}^{p^{\lambda(1)+\lambda(2)}}x^{ip^{\lambda(1)+\lambda(2)}}.
$$
Let $S_2=S_1[\frac{y_1}{x^{p^{\lambda(1)}}},\alpha_1]_{(x,y_2)}\subset K$. We have a  local $k$-algebra
homomorphism $\pi_2:S_2\rightarrow \overline k[[x]]$ which extends $\pi_1$.  We have
$S_2/m_{S_2}\cong k(\alpha_1,\alpha_2^{p^{\lambda(1)}})$, so that
$$
[S_2/m_{S_2}:S_1/m_{S_1}]=[ k(\alpha_1,\alpha_2^{p^{\lambda(1)}}):k(\alpha_1)]=p^{\lambda(2)}.
$$
We iterate the above construction, defining for $i\ge 2$,
$$
\begin{array}{lll}
y_i&=&\left(\frac{y_{i-1}}{x^{p^{\lambda(1)+\cdots+\lambda(i-1)}}}\right)^{p^{\lambda(i)}}-\alpha_i^{p^{\lambda(1)+\cdots+\lambda(i)}}\\
&=&\sum_{j=1}^{\infty}\alpha_{j+i}^{p^{\lambda(1)+\cdots+\lambda(i)}}
x^{jp^{\lambda(1)+\cdots+\lambda(i)}}
\end{array}
$$
where $p^{\lambda(i)}\in \NN$ is the smallest natural number   such that
$$
\alpha_i^{p^{\lambda(1)+\cdots+\lambda(i)}}\in k(\alpha_1,\alpha_2^{p^{\lambda(1)}},\ldots,\alpha_{i-1}^{p^{\lambda(1)+\cdots+\lambda(i-2)}}).
$$
Define
$$
S_i=S_{i-1}\left[\frac{y_{i-1}}{x^{p^{\lambda(1)+\cdots+\lambda(i-1)}}},
\alpha_{i-1}^{p^{\lambda(1)+\cdots+\lambda(i-2)}}\right]
_{(x,y_i)},
$$
to construct an infinite commutative diagram of regular local rings, which are contained in $K$,
$$
\begin{array}{lllllllll}
S_0&\rightarrow&S_1&\rightarrow&\cdots&\rightarrow&S_i&\rightarrow&\cdots\\
\pi_0\downarrow&&\pi_1\downarrow&&&&\pi_i\downarrow&&\\
\overline k[[x]]&=&\overline k[[x]]&=&\cdots&=&\overline k[[x]]&=&\cdots
\end{array}
$$
We have
\begin{equation}\label{eq30}
S_i/m_{S_i}\cong S_{i-1}/m_{S_{i-1}}[\alpha_i^{p^{\lambda(1)+\cdots+\lambda(i-1)}}]
\end{equation}
and
$$
[S_i/m_{S_i}:S_{i-1}/m_{S_{i-1}}]=p^{\lambda(i)}.
$$

For all $i$, the field
$$
k_i:=k(\alpha_1,\alpha_2^{p^{\lambda(1)}},
\ldots,\alpha_{i-1}^{p^{\lambda(1)+\cdots+\lambda(i-2)}})\subset S_i,
$$
and
$$
S_i/m_{S_i}\cong k_i[\alpha_i^{p^{\lambda(1)+\cdots+\lambda(i-1)}}].
$$
We now construct a sequence
$$
R_0\rightarrow R_1\rightarrow\cdots\rightarrow R_i\rightarrow\cdots
$$
of birationally equivalent regular local rings such that there is
a commutative diagram of local $k$-algebra homomorphisms
$$
\begin{array}{lllllllll}
R_0&\rightarrow &R_1&\rightarrow &\cdots&\rightarrow&R_i&\rightarrow&\cdots\\
\downarrow&&\downarrow&&&&\downarrow&&\\
S_0&\rightarrow &S_1&\rightarrow &\cdots&\rightarrow&S_i&\rightarrow&\cdots
\end{array}
$$
satisfying
$$
m_{R_i}S_i=m_{S_i}\text{ and }S_i/m_{S_i}\cong R_i/m_{R_i}
$$
for all $i$. The vertical arrows are inclusions.

This is certainly the case for $R_0=S_0$, so we suppose that we have constructed the sequence
out to $R_i\rightarrow S_i$, and show that we may extend it to
$R_{i+1}\rightarrow S_{i+1}$.

We have
$$
\alpha_{i+1}^{p^{\lambda(1)+\cdots+\lambda(i+1)}}\in
k(\alpha_1,\alpha_2^{p^{\lambda(1)}},\ldots,\alpha_i^{p^{\lambda(1)+\cdots+\lambda(i-1)}})\cong R_i/m_{R_i}.
$$
Thus there exists $\phi\in R_i$ such that the class  of $\phi$ in $R_i/m_{R_i}$ is
$$
[\phi]=\alpha_{i+1}^{p^{\lambda(1)+\cdots+\lambda(i+1)}}.
$$
Our assumptions $m_{R_i}S_i=m_{S_i}$ and $S_i/m_{S_i}\cong R_i/m_{R_i}$ imply that
\begin{equation}\label{eqz}
m_{R_i}^n/m_{R_i}^{n+1}\cong m_{S_i}^n/m_{S_i}^{n+1}
\end{equation}
as $R_i/m_{R_i}$ vector spaces for all $n\in \NN$.

By (\ref{eqz}), there exists $z_i\in R_i$ such that
$$
z_i=y_i+h
$$
with $h\in m_{S_i}^{2+p^{\lambda(1)+\cdots+\lambda(i)}}$. We then have that $m_{R_i}=(x,z_i)$, since $m_{R_i}/m_{R_i}^2\cong m_{S_i}/m_{S_i}^2$ as $R_i/m_{R_i}$ vector spaces, and by
Nakayama's Lemma. Now

$$
\begin{array}{lll}
\left(\frac{z_i}{x^{p^{\lambda(1)+\cdots+\lambda(i)}}}\right)^{p^{\lambda(i+1)}}
&=&\left(\frac{y_i}{x^{p^{\lambda(1)+\cdots+\lambda(i)}}}\right)^{p^{\lambda(i+1)}}
+\left(\frac{h}
{x^{p^{\lambda(1)+\cdots+\lambda(i)}}}\right)^{p^{\lambda(i+1)}}\\
&=&\left(\frac{y_i}{x^{p^{\lambda(1)+\cdots+\lambda(i)}}}\right)^{p^{\lambda(i+1)}}+xh'
\end{array}
$$
for some $h'\in S_i[\frac{y_i}{x^{p^{\lambda(1)+\cdots+\lambda(i)}}}]$.
$$
\left(\frac{z_i}{x^{p^{\lambda(1)+\cdots+\lambda(i)}}}\right)^{p^{\lambda(i+1)}}
-\phi\in
R_i\left[\frac{z_i}{x^{p^{\lambda(1)+\cdots+\lambda(i)}}}\right]\subset
S_i\left[\frac{y_i}{x^{p^{\lambda(1)+\cdots+\lambda(i)}}}\right]
$$

has residue
$$
\left(\frac{y_i}{x^{p^{\lambda(1)+\cdots+\lambda(i)}}}\right)^{p^{\lambda(i+1)}}-\alpha_{i+1}^{p^{\lambda(1)+\cdots+\lambda(i+1)}}
$$
in $S_{i+1}/xS_{i+1}\cong
S_i/m_{S_i}\left[\frac{y_i}{x^{p^{\lambda(1)+\cdots+\lambda(i)}}}\right]$. Thus

$$
m_{S_{i+1}}\cap
R_i\left[\frac{z_i}{x^{p^{\lambda(1)+\cdots+\lambda(i)}}}\right]
=(x,\left(\frac{z_i}{x^{p^{\lambda(1)+\cdots+\lambda(i)}}}\right)^{p^{\lambda(i+1)}}-\phi).
$$
Let
$$
R_{i+1}=R_i\left[\frac{z_i}{x^{p^{\lambda(1)+\cdots+\lambda(i)}}}\right]_{(x,\left(\frac{z_i}{x^{p^{\lambda(1)+\cdots+\lambda(i)}}}\right)^{p^{\lambda(i+1)}}-\phi)}.
$$

We have $m_{R_{i+1}}S_{i+1}=m_{S_{i+1}}$ (by Nakayama's Lemma) and $R_{i+1}/m_{R_{i+1}}\cong S_{i+1}/m_{S_{i+1}}$.

 We have  factorizations $g(x,y)=x^{\beta_i}g_i$ where $\beta_i\in\NN$ and
$g_i\in R_i$ is either irreducible or a unit. $g_i$ is a strict transform of $g$ in
$R_i$. Since $\pi_i(x)\ne 0$, we have that $g_i$ is contained in the kernel of the map
$R_i\rightarrow S_i\stackrel{\pi}{\rightarrow} \overline k[[x]]$,
 and thus the ideal $(g_i)$ is the (nontrivial) kernel of $R_i\rightarrow \overline k[[x]]$. In particular, $g_i\in M_{R_i}$ for all $i$.

Each  extension $R_i\rightarrow R_{i+1}$ can be factored as a sequence of $p^{\lambda(1)+\cdots+\lambda(i)}$ birationally equivalent regular local rings,
each of which is a quadratic transform (the blow up of the maximal ideal followed by localization). The $j$-th local ring with $j<p^{\lambda(1)+\cdots+\lambda(i)}$, has the maximal ideal $(x,\frac{z_i}{x^j})$.

By embedded resolution of plane curve singularities (\cite{Ab1}, \cite{C}, \cite{O}), we obtain that there exists $i_0$ such that $g=0$ is a simple normal crossings divisor in $\mbox{Spec}(R_i)$ for all $i\ge i_0$, so that
$x,g_i$ is  a regular system of parameters in $R_i$ for all $i\ge i_0$. Thus
$$
R_{i+1}=R_i\left[\frac{z_i}{x^{p^{\lambda(1)+\cdots+\lambda(i)}}}\right]_{(x,\frac{g_i}{x^{p^{\lambda(1)+\cdots+\lambda(i)}}})}
$$
for all $i\ge i_0$, and
$$
S_{i+1}/m_{S_{i+1}}\cong R_{i+1}/m_{R_{i+1}}\cong R_i/m_{R_i}\cong S_i/m_{S_i}
$$
for all $i\ge i_0$. Thus $\lambda(i)=0$
for all $i\ge i_0+1$.

Let $$
M=k(\alpha_1,\alpha_2^{p^{\lambda(1)}},\ldots,\alpha_{i_0}^{p^{\lambda(1)+\cdots+\lambda(i_0-1)}})\cong S_{i_0}/m_{S_{i_0}}.
$$
From (\ref{eq30}), we see that $L^{p^{\lambda(1)+\cdots+\lambda(i_0)}}\subset M$. Since $M$ is a finitely generated  purely inseparable extension of $k$, there exists $r\in\NN$ such that $M^{p^r}\subset k$. Thus
$L^{p^{\lambda(1)+\cdots+\lambda(i_0)+r}}\subset k$.

\end{proof}

\begin{Theorem}\label{fin} Suppose that $k$ is a field of characteristic $p\ge0$,
with algebraic closure
$\overline k$. Let $\overline k((x))$ be the field  of formal Laurent series with coefficients in $\overline k$. Suppose that
$$
\sigma(x)=\sum_{i=d}^{\infty}\alpha_ix^i\in \overline k((x))
$$
where  $d\in \ZZ$ and $\alpha_i\in \overline k$ for all $i$.
Let $L= k(\{\alpha_i\mid i\in \NN\})$.       Then $\sigma(x)$ is algebraic over $k((x))$ if and only if
there exists $n\in\NN$ such that  $[kL^{p^n}:k]<\infty$, where $kL^{p^n}$
is the compositum of $k$ and $L^{p^n}$ in $\overline k$.

\end{Theorem}

\begin{proof}
First suppose that $[kL^{p^n}:k]<\infty$ for some $n$. After possibly replacing
$n$ with a larger value of $n$, we may assume that $kL^{p^n}$ is separable over $k$.
Then $\sigma(x)^{p^n}$
is algebraic over $k((x))$ by Theorem \ref{main}, and thus $\sigma(x)$ is algebraic over $k((x))$.

Now suppose that $\sigma(x)$ is algebraic over $k((x))$. Let $M$
be the separable closure of $k$ in $L$. Then $\sigma(x)$ is
algebraic over $M((x))$. Since $L$ is a purely inseparable
extension of $M$, it follows from Theorem \ref{insep} that
$\tau(x)=\sigma(x)^{p^n}\in M[[x]]$ for some $n\in\NN$. Since
$\tau(x)$ is algebraic over $k((x))$, we have that
$[kL^{p^n}:k]<\infty$ by Theorem \ref{main}.
\end{proof}

\begin{Corollary} Suppose that $k$ is a  field of characteristic $p\ge 0$ such that $k$ is a finitely generated extension of a perfect field, with algebraic closure
$\overline k$. Let $\overline k((x))$ be the field  of formal Laurent series with coefficients in $\overline k$. Suppose that
$$
\sigma(x)=\sum_{i=d}^{\infty}\alpha_ix^i\in \overline k((x))
$$
where  $d\in \ZZ$ and $\alpha_i\in \overline k$ for all $i$. Let $L=k(\{\alpha_i\mid i\in \NN\})$.  Then $\sigma(x)$ is algebraic over $k((x))$ if and
only if
$$
[L:k]<\infty.
$$
\end{Corollary}

\begin{proof}
If $[L:k]<\infty$, then $\sigma(x)$ is algebraic over $k((x))$ by Theorem
\ref{fin}.

Suppose that $\sigma(x)$ is algebraic over $k((x))$. By
assumption, there exists a perfect field $F$ and $s_1,\ldots,
s_r\in k$ such that $k=F(s_1,\ldots,s_r)$. By Theorem \ref{fin},
there exists $n$ such that $[kL^{p^n}:k]<\infty$. Thus
$kL^{p^n}=F(s_1,\ldots,s_r,\beta_1,\ldots,\beta_s)$ where
$\beta_1,\ldots,\beta_s\in kL^{p^n}$ are algebraic over $k$. Thus
$$
L\subset F(s_1^{\frac{1}{p^n}},\ldots,s_r^{\frac{1}{p^n}},\beta_1^{\frac{1}{p^n}},
\ldots,\beta_s^{\frac{1}{p^n}}).
$$
Now
$$
[F(s_1^{\frac{1}{p^n}},\ldots,s_r^{\frac{1}{p^n}}):F(s_1,\ldots,s_r)]<\infty
$$
and since $\beta_1,\ldots,\beta_s$ are algebraic over $F(s_1,\ldots,s_r)$,
$$
[F(s_1^{\frac{1}{p^n}},\ldots,s_r^{\frac{1}{p^n}},\beta_1^{\frac{1}{p^n}},
\ldots,\beta_s^{\frac{1}{p^n}}):
F(s_1^{\frac{1}{p^n}},\ldots,s_r^{\frac{1}{p^n}})]<\infty.
$$
Thus
$$
[L:k]\le [F(s_1^{\frac{1}{p^n}},\ldots,s_r^{\frac{1}{p^n}},\beta_1^{\frac{1}{p^n}},
\ldots,\beta_s^{\frac{1}{p^n}}):k]<\infty.
$$
\end{proof}

\section{Series in several variables}


We will now generalize theorem \ref{fin} to higher dimensions.

Denote by $X$ an $n$-dimensional indeterminate vector
$(x_1,x_2,\dots,x_n)$ and by $I$ an $n$-dimensional exponent vector
$(i_1,i_2,\dots,i_n)\in\mathbb{N}^n$. Then for $1\le l\le n$ write
$X_l=(x_1,x_2,\dots,x_l)$, $I_l=(i_1,i_2,\dots,i_l)$ and
$X^I_l=X_l^{I_l}=x_1^{i_1}x_2^{i_2}\cdots x_l^{i_l}$. If $E$ is a
field denote by $E[[X]]$ the formal power series ring in $n$
variables with coefficients in $E$ and by $E((X))$ the quotient
field of $E[[X]]$. Also denote by $E^c$ the perfect closure of $E$
and by $\overline{E}$ the algebraic closure of $E$.

\begin{Lemma}\label{separation}
Suppose that $E$ is a field and $F$ is a field extension of $E$. Let
$$
\s=\sum_{I\in\mathbb{N}^n}\a_IX^I\in F[[X]], {\text{ with }} \a_I\in
F,
$$
be a formal power series in $n$ variables with coefficients in $F$.
For any $1\le l\le n$ and $I\in\mathbb{N}^n$ define the following
power series in 1 variable with coefficients in $F$
$$
a_{I,l}=\sum_{j=0}^{\infty}\a_Jx_l^j, {\text{ where }}
J=(i_1,i_2,\dots,i_{l-1},j,i_{l+1},\dots,i_n).
$$
Then $\s$ is algebraic over $E((X))$ implies $a_{I,l}$ is algebraic
over $E((x_l))$.
\end{Lemma}

\begin{proof}
We use induction on the number of variables. If $n=1$ the
statement is trivial. Suppose that $n>1$. After possibly permuting
the variables we may assume that $l=1$. Write
$X_{n-1}=(x_1,\dots,x_{n-1})$ and for all $m\in\mathbb{N}$
consider the power series in $n-1$ variables
$$
\d_m=\sum_{R\in\mathbb{N}^n,\,r_n=m}\a_RX^R_{n-1}=
\sum_{R\in\mathbb{N}^n,\,r_n=m}\a_Rx_1^{r_1}x_2^{r_2}\cdots
x_{n-1}^{r_{n-1}}.
$$
If $\d_{i_n}$ is algebraic over $E((X_{n-1}))$ it will follow from
the inductive hypothesis that $a_{I,1}$ is algebraic over
$E((x_1))$. We will show that $\d_m$ is algebraic over
$E((X_{n-1}))$ for all $m\in\mathbb{N}$.

Consider the algebraic dependency relation for $\s$ over $E((X))$
$$
c_t(X)\s^t+c_{t-1}(X)\s^{t-1}+\dots+c_1(X)\s+c_0(X)=0.
$$
By clearing the denominators we may assume that $c_j\in E[[X]]$ for
all $0\le j\le t$. Let $g$ be the highest power of $x_n$ that
divides $c_j$ for all $j$. Set
$c'_j=(x_n^{-g}c_j)(x_1,x_2,\dots,x_{n-1},0)$. Then $c'_j\in
E[[X_{n-1}]]$ and the following equation holds
$$
c'_t(X_{n-1})\d^t_0+c'_{t-1}(X_{n-1})\d^{t-1}_0+\dots+c'_1(X_{n-1})\d_0+c'_0(X_{n-1})=0,
$$
where $c'_j\neq 0$ for some $0\le j\le t$. Thus $\d_0$ is algebraic
over $E((X_{n-1}))$.

Set $\s_1=x_{n}^{-1}(\s-\d_0)$. Then $\s_1\in F[[X]]$ and it is
algebraic over $E((X))$. Arguing as above we get that $\d_1$ is
algebraic over $E((X_{n-1}))$. In general we define
$\s_m=x_n^{-1}(\s_{m-1}-\d_{m-1})$ recursively for all
$m\in\mathbb{N}$ and use $\s_m$ to prove that $\d_m$ is algebraic
over $E((X_{n-1}))$.
\end{proof}

\begin{Theorem}\label{multi}
Suppose that $k$ is a field of characteristic $p\ge 0$. Suppose
that
$$
\s=\sum_{I\in\mathbb{N}^n}\alpha_I X^I\in \overline k[[X]],
{\text{ with }} \a_I\in\overline{k}
$$
is a formal power series in $n$ variables with coefficients in
$\overline{k}$. Let $L=k(\{\alpha_I\mid I\in\mathbb{N}^n\})$ be
the extension field of $k$ generated by the coefficients of $\s$.
Then $\s$ is algebraic over $k((X))$ if and only if there exists
$r\in\mathbb{N}$ such that $[kL^{p^r}:k]<\infty$, where $kL^{p^r}$
is the compositum of $k$ and $L^{p^r}$ in $\overline k$.
\end{Theorem}

\begin{proof}
First suppose that there exists $r\in\mathbb{N}$ such that
$[kL^{p^r}:k]<\infty$. After possibly increasing $r$ we may assume
that $kL^{p^r}$ is a separable extension of $k$. Let $M$ be a
finite Galois extension of $k$ which contains $kL^{p^r}$. Notice
that $kL^{p^r}=k(\{\a_I^{p^r}\mid I\in\mathbb{N}^n\})$ and,
therefore $\s^{p^r}\in M[[X]]$. Let $G$ be the Galois group of $M$
over $k$. $G$ acts naturally by $k$ algebra isomorphisms on
$M[[X]]$, and the invariant ring of the action is $k[[X]]$. Let
$f(y)=\prod_{\tau\in G}(y-\tau(\s^{p^r}))\in M[[X]][y]$. Since $f$
is invariant under the action of $G$, $f(y)\in k[[X]][y]$. Since
$f(\s^{p^r})=0$, we have that $\s$ is algebraic over $k[[X]]$.

To prove the other implication we use induction on the number of
variables. When $n=1$ the statement follows from theorem
\ref{fin}. Assume that $n>1$.

For all $I\in\mathbb{N}^n$ let
$$
a_I=\sum_{j=0}^{\infty}\a_Jx_n^j, {\text { with }}
J=(i_1,i_2,\dots,i_{n-1},j),
$$
be a power series in 1 variable with coefficients in
$\overline{k}$. If $K=k((x_n))$ then by lemma \ref{separation}
$a_I$ is algebraic over $K$ for all $I\in\mathbb{N}^n$. Then
$$
\s=\sum_{\{I\in\mathbb{N}^n\,\mid\, i_n=0\}}a_IX^I_{n-1}
$$
is a series in $n-1$ variables with coefficients in
$\overline{K}$. By the inductive hypothesis there exists
$N\in\mathbb{N}$ and $r\in\mathbb{N}$ such that $K(\{a_I^{p^r}\mid
I\in\mathbb{N}^n\})=K(a^{p^r}_{I_1},a^{p^r}_{I_2},\dots,a^{p^r}_{I_N})$.
Thus, for all $I\in\mathbb{N}^n$ we have $a_I^{p^r}$ is a
polynomial in $a^{p^r}_{I_1},a^{p^r}_{I_2},\dots,a^{p^r}_{I_N}$
with coefficients in $K$.

Fix $I\in\mathbb{N}$, if $j\in\mathbb{N}$ set
$J=(i_1,i_2,\dots,i_{n-1},j)$ and write
$$
\sum_{j=0}^{\infty}\a_J^{p_r}x_n^{jp^r}=
a^{p^r}_I=\sum_{S\in\{0,1,\dots,T\}^N}(\sum_{m=-M_S}^{\infty}\g_{S,m}x_n^m)
(a_{I_1}^{p^r})^{s_1}(a_{I_2}^{p^r})^{s_2}\cdots
(a_{I_N}^{p^r})^{s_N},
$$
where $T\in\mathbb{N}$, $S=(s_1,s_2,\dots,s_N)$ is an index
vector, $M_S\in\mathbb{N}$ and $\g_{S,m}\in k$ for all $S$ and
$m$. This implies that for all $I\in\mathbb{N}$ and
$j\in\mathbb{N}$, $\a_J^{p^r}$ is a polynomial in the coefficients
of power series $a_{I_1}^{p^r},a_{I_2}^{p^r},\dots,a_{I_N}^{p^r}$
over $k$. Moreover, for all $r'\ge r$ we also have $\a_J^{p^{r'}}$
is a polynomial in the coefficients of power series
$a_{I_1}^{p^{r'}},a_{I_2}^{p^{r'}},\dots,a_{I_N}^{p^{r'}}$ over
$k$. Thus $kL^{p^{r'}}$ is the field extension of $k$ generated by
the coefficients of power series
$a_{I_1}^{p^{r'}},a_{I_2}^{p^{r'}},\dots,a_{I_N}^{p^{r'}}$.

Applying theorem \ref{fin} to each of the series
$a_{I_1},a_{I_2},\dots,a_{I_N}$ we see that there exists
$R\in\mathbb{N}$ such that $kL^{p^R}$ is finitely generated over
$k$.

\end{proof}

Similarly to the case of one variable we deduce the following
corollary

\begin{Corollary}\label{1}
Suppose that $k$ is a field of characteristic $p\ge 0$ such that
$k$ is a finitely generated extension of a perfect field. Suppose
that
$$
\s=\sum_{I\in\mathbb{N}^n}\alpha_I X^I\in \overline k[[X]],
{\text{ with }} \a_I\in\overline{k}
$$
is a formal power series in $n$ variables with coefficient in
$\overline{k}$. Let $L=k(\{\alpha_I\mid I\in\mathbb{N}^n\})$ be
the extension field of $k$ generated by the coefficients of $\s$.
Then $\s$ is algebraic over $k((X))$ if and only if
$[L:k]<\infty$.
\end{Corollary}

Also notice that if $E$ is a field of characteristic $p\ge 0$ and
$a$ is separable algebraic over $E$ then for all $r\in\mathbb{N}$
we have $E[a^{p^r}]=E[a]$. Thus if $F$ is a separable extension of
$E$, $EF^{p^r}=F$ for all $r\in\mathbb{N}$. So we have the
following statement in case of separable extensions.

\begin{Corollary}\label{2}
Suppose that $k$ is a field of characteristic $p\ge 0$. Suppose
that
$$
\s=\sum_{I\in\mathbb{N}^n}\alpha_I X^I\in \overline k[[X]],
{\text{ with }} \a_I\in\overline{k}
$$
is a formal power series in $n$ variables with coefficient in
$\overline{k}$. Let $L=k(\{\alpha_I\mid I\in\mathbb{N}^n\})$ be
the extension field of $k$ generated by the coefficients of $\s$.
Suppose that $L$ is separable over $k$. Then $\s$ is algebraic
over $k((X))$ if and only if $[L:k]<\infty$.
\end{Corollary}

\section{Valuations whose rank increases under completion}

Suppose that $K$ is a field and $V$ is a valuation ring of $K$. We will say that
{\it the rank of $V$ increases under completion} if there exists an analytically normal local domain $T$ with quotient field $K$ such that $V$ dominates $T$
and there exists an extension of $V$ to a valuation ring of the quotient field of $\hat T$
which dominates $\hat T$ which has higher rank than the rank of $V$.

Suppose that $V$ dominates an excellent local ring $R$ of dimension 2. Then
by resolution of surface singularities \cite{L}, there exists a regular local ring $R_0$ and a birational extension
$R\rightarrow R_0$ such that $V$ dominates $R_0$. Let
\begin{equation}\label{eq1}
R_0\rightarrow R_1\rightarrow\cdots\rightarrow R_n\rightarrow\cdots
\end{equation}
be the infinite sequence of regular local rings obtained by blowing up
the maximal ideal of $R_i$ and localizing at the center of $V$.
Since $R$ has dimension 2, we have that $V=\cup_{i=0}^{\infty}R_i$ (as is shown in \cite{Ab2}), and
thus $V/m_V=\cup_{i=0}^{\infty}R_i/m_{R_i}$. We see that $V/m_V$ is countably generated over $R/m_R$.

Suppose that the rank of $V$ increases under completion. Then there exists $n$
such that for all $i\ge n$, there exists a valuation ring $V_1$ of the quotient field of the regular
local ring $\hat R_i$ which extends $V$, dominates $\hat R_i$, and has rank larger than 1. By the Abhyankar inequality (\cite{Ab2} or Proposition 3 of Appendix 2 \cite{ZS}),
we have that $R_i$ has dimension 2, $V_1$ is discrete of rank 2, and
$V_1/m_{V_1}$ is algebraic over $\hat R_i/m_{\hat R_i}$. Thus $V/m_V$ is algebraic over $R/m_R$ and $V$ is discrete of rank 1.

  It was shown by Spivakovsky \cite{S} in the case that $R/m$ is algebraically closed that the converse holds, giving
the following simple characterization.

\begin{Theorem}\label{TheoremV1}(Spivakovsky \cite{S})
Suppose that $V$ dominates an excellent two dimensional  local ring $R$ who residue field
$R/m_R$ is algebraically closed. Then the rank of $V$ increases under completion
if and only if $\mbox{dim}_R(V)=0$ and $V$ is discrete of rank 1.
\end{Theorem}

The condition that the transcendence degree $\mbox{dim}_R(V)$ of $V/m_V$ over
$R/m_R$ is zero is just the statement that $V/m_V$ is algebraic over $R/m_R$.
In the case that $R/m_R$ is algebraically closed, $\mbox{dim}_R(V)=0$
if and only if $V/m_V=R/m_R$.

Using a similar method to that used in the proof of our algebraicity theorem on power series, Theorem \ref{main}, we prove the
following extension of Theorem \ref{TheoremV1}.

\begin{Theorem}\label{TheoremV2}
Suppose that $V$ is a valuation ring of a field $K$, and $V$ dominates an excellent two dimensional  local domain $R$ whose quotient field is $K$. Then the rank of $V$ increases under completion
if and only if $V/m_V$ is finite over $R/m_R$ and $V$ is discrete of rank 1.
\end{Theorem}

\begin{proof}
First assume that the rank of $V$ increases under completion.
Consider the sequence (\ref{eq1}).
We observed above after (\ref{eq1}) that $V/m_V$ is algebraic over $R/m_R$
and $V$ is discrete of rank 1. Further, there exists $R_i$ and a valuation
$V_1$ of the quotient field  of $\hat R_i$ which dominates $\hat R_i$ whose intersection with the quotient field $K$ of $R$ is $V$, and the rank of $V_1$ is 2. Without loss of generality, we may assume that $R_i=R_0$.

For $i\ge 0$, let $p(R_i)_{\infty}$ be the (nontrivial) prime ideal in $\hat R_i$ of cauchy sequences whose value is greater than $n$ for any $n\in\NN$
(Section 5 of \cite{CG}). Since $\hat R_i$ is a two dimensional regular local ring, $p(R_i)_{\infty}$ is generated by an irreducible element in $\hat R_i$ for all $i$. Let $f$ be a generator of $p(R_0)_{\infty}$. By resolution of plane curve singularities (\cite{Ab1}, \cite{C}, \cite{O}), there exists $i$ in the sequence (\ref{eq1}) such that
$f=h_if_i$, where  $h_i\in R_i$ is such that $h_i=0$ is supported on the
exceptional locus of  $\mbox{Spec}(R_i)\rightarrow \mbox{Spec}(R)$,
and $f_i\in \hat R_i$ is such that $\hat R_i/f_i\hat R_i$ is a regular local ring.
We necessarily have that $p(R_i)_{\infty}=f_i\hat R_i$. Again, without loss of generality, we may assume that $i=0$. Let $T_0=\hat R_0$, and let
$$
T_0\rightarrow T_1\rightarrow \cdots \rightarrow T_n\rightarrow\cdots
$$
be the infinite sequence of regular local rings obtained by blowing up the maximal ideal of the regular local ring $T_i$ and localizing at the center of $V_1$.  We then have a commutative diagram
\begin{equation}\label{eq5}
\begin{array}{lllllllllll}
R_0&\rightarrow &R_1&\rightarrow&\cdots&\rightarrow &R_i&\rightarrow&\cdots\\
\downarrow&&\downarrow&&&&\downarrow&&\\
T_0&\rightarrow &T_1&\rightarrow&\cdots&\rightarrow &T_i&\rightarrow&\cdots\\
\downarrow&&\downarrow&&&&\downarrow&&\\
\hat R_0&\rightarrow &\hat R_1&\rightarrow&\cdots&\rightarrow &\hat R_i&\rightarrow&\cdots\\
\end{array}
\end{equation}

There exists $x\in R_0$ such that $x,f_0$ is a regular system of parameters in $T_0$.  Thus $T_1=T_0[\frac{f_0}{x}]_{(x,\frac{f_0}{x})}$. Define
$f_i=\frac{f_0}{x^i}$ for $i\ge 1$. Then $T_i=T_0[f_i]_{(x,f_i)}$ and
$p(R_i)_{\infty}=f_i\hat R_i$ for all $i\ge 0$. Thus $R_i/m_{R_i}\cong T_i/m_{T_i}\cong
T_0/(x,f_0)\cong R_0/m_{R_0}$ for all $i$. Since $V/m_V=\cup_{i\ge 0} R_i/m_{R_i}
=R_0/m_{R_0}$ and $R_0/m_{R_0}$ is finite over $R/m_R$, we have the conclusions of
the theorem.

Now assume that $V/m_V$ is finite over $R/m_R$ and $V$ is discrete of rank 1.
Consider the sequence (\ref{eq1}). There exists $i$ such that $R_i/m_{R_i}=V/m_V$. Without loss of generality, we may assume that $R=R_i$.
Let $\nu$ be  a valuation of $K$ such that $V$ is the valuation ring of $\nu$.
We may also assume that there are regular parameters $x,y$ in $R$ such that
$\nu(x)=1$ generates the value group $\ZZ$ of $\nu$.  Let $\pi:R\rightarrow R/m_R=V/m_V$ be the residue map. Let $y_0=y$.
There exists $n_0\in\NN$ such that $\nu(y)=n_0$. Let $\alpha_0\in R$ be such that $\pi(\alpha_0)=[\frac{y}{x^{n_0}}]\in V/m_V$. Let $y_1=y-\alpha_0x^{n_0}$,
and let $n_1=\nu(y_1)$. We have $n_1>n_0$. Iterate, to construct $y_i\in R$
and $n_i\in \NN$ with $\nu(y_i)=n_i$ for
$i\in \NN$ by choosing $\alpha_i\in R$ such that $y_{i+1}=y_i-\alpha_ix^{n_i}$
satisfies $n_{i+1}>n_i$. Thus $\{y_i\}$ is a Cauchy sequence in $R$. Let $\sigma$ be the limit of $\{y_i\}$ in $\hat R$.  Let $\hat \nu$ be an extension of $\nu$ to the quotient field of $\hat R$ which dominates $\hat R$.
Then $\hat\nu(\sigma)>n$ for all $n\in\NN$, so that $\hat \nu$ has rank $2>1$,
and we see that the rank of $V$ increases under completion.
\end{proof}

We see that the condition that $V/m_V$ is finite over $R/m_R$ thus divides the
class of discrete rank 1  valuation rings with $\mbox{dim}_R(V)=0$ into two subclasses, those
whose rank increases under completion ($[V/m_V:R/mR]<\infty$), and those whose rank does not increase ($[V/m_V:R/mR]=\infty$). We have the following precise
characterization of when this division into subclasses is nontrivial.

\begin{Corollary}\label{Cor} Suppose that $R$ is an excellent two dimensional local ring.
Then there exists a rank 1 discrete valuation ring $V$ of the quotient field of $R$ which dominates
$R$ such that $\mbox{dim}_R(V)=0$ and the rank of $V$ does not increase under completion if and
only if $[\overline k:k]=\infty$, where $\overline k$ is the algebraic closure of $k=R/m_R$.
\end{Corollary}

\begin{proof} Suppose that  $[\overline k:k]<\infty$, and $V$ is a rank 1 discrete
 valuation ring  of the quotient field of $R$
which dominates $R$ such that $\mbox{dim}_R(V)=0$. Then
$$
[V/m_V:k]\le [\overline k:k]<\infty.
$$
Thus the rank of $V$ must increase under completion by Theorem \ref{TheoremV2}.

Now suppose that $[\overline k:k]=\infty$. We will construct a rank 1 discrete
valuation ring $V$ of the quotient field of $R$ which dominates $R$ such that
$\mbox{dim}_R(V)=0$ and the rank of $V$ does not increase under completion.

There exists a two dimensional regular local ring $R_0$ which birationally
dominates $R$. We have $[\overline k:R_0/m_{R_0}]=\infty$.
Let $x,y_0$ be a regular system of parameters in $R_0$. We will inductively construct an infinite birational sequence of regular local rings
$$
R_0\rightarrow R_1\rightarrow\cdots\rightarrow R_i\rightarrow\cdots
$$
such that $R_i$ has a regular system of parameters $x,y_i$ and
$[R_i/m_{R_i}:R_{i-1}/m_{R_{i-1}}]>1$ for all $i$.
Suppose that we have defined the sequence out to $R_i$. Choose
$\alpha_{i+1}\in\overline k-R_i/m_{R_i}$. Let $h_{i+1}(t)$ be the minimal polynomial of $\alpha_{i+1}$ in the polynomial ring $R_i/m_{R_i}[t]$.
We have an isomorphism
$$
R_i\left[\frac{m_{R_i}}{x}\right]/x R_i\left[\frac{m_{R_i}}{x}\right]
\cong R_i/m_{R_i}\left[\frac{y_i}{x}\right].
$$
Let $y_{i+1}$ be a lift of $h_{i+1}(\frac{y_i}{x})$ to $R_i\left[\frac{m_{R_i}}{x}\right]$. Let
$$
R_{i+1}=R_i\left[\frac{m_{R_i}}{x}\right]_{(x,y_{i+1})}.
$$
We have that $R_{i+1}/m_{R_{i+1}}\cong R_i/m_{R_i}(\alpha_{i+1})$.

Let $V=\cup_{i=0}^{\infty}R_i$. $V$ is a valuation ring which dominates $R$
(as is shown in \cite{Ab2}).
$V/m_V=\cup_{i=0}^{\infty}R_i/m_{R_i}$ so that $\mbox{dim}_R(V)=0$
and $[V/m_V:k]=\infty$.

$V$ must have rank 1 since $[V/m_V:k]=\infty$ (for instance by the Abhyankar inequality, \cite{Ab2} or  Proposition 3 \cite{ZS}). By our construction, $\nu(x)\le\nu(f)$ for
any $f\in m_V=\cup_{i=1}^{\infty}m_{R_i}$. Thus the value group of $V$ is
discrete. Since $[V/m_V:k]=\infty$, by  Theorem \ref{TheoremV2} the rank of $V$ does not increase under completion.
\end{proof}

When a valuation ring $V$ with quotient field $K$ is equicharacteristic and  discrete of rank 1, it can be  explicitly
described by a representation in a power series ring in one variable over
the residue field of $V$. In fact, since $V$ is discrete of rank 1, it is
Noetherian (Theorem 16, Section 10, Chapter VI \cite{ZS}). As $V$ is equicharacteristic, the $m_V$-adic completion $\hat V$ of $V$ has a coefficient
field $L$ by Cohen's theorem, and thus $\hat V\cong L[[t]]$ is a power series ring in one variable over $L\cong V/m_V$. We have $V=K\cap \hat V$. The subtlety of this statement is that if $k$ is a subfield of $K$ contained in $V$ such that
$V/m_V$ is not separably generated over $k$, then there may not exist a
coefficient field $L$ of $\hat V$ which contains $k$.

Although the completion of a rank 1 valuation ring is a power series ring,
in positive characteristic,
the  valuation determined by  associating to a system of parameters specific  power series may not be easily recognizable from a series representation of the valuation ring. This can be seen from
 the contrast of the  conclusions of
Theorem \ref{varfin} with the results of this section.  The finiteness condition
$[L:k]<\infty$ of the coefficient field of a series over a base field $k$
does not characterize algebraicity of a series in positive characteristic,
while the corresponding finiteness condition on residue field extensions does
characterize algebraicity in the case of valuations dominating a local ring of Theorem \ref{TheoremV2}. We illustrate this distinction in the following example.

\begin{Example}\label{example2} The valuation induced by the series of
Example \ref{example}, whose coefficient field is infinitely
algebraic over the base field $k$, has a residue field which is
finite over $k$.
\end{Example}

\begin{proof} With notation of Example \ref{example},
we have a $k$-algebra homomorphism
$$
R=k[u,v]_{(u,v)}\stackrel{\pi}{\rightarrow} \overline k[[x]]
$$
defined by the substitutions
$$
u=x, v=\sigma(x)=\sum_{i=1}^{\infty} t_i^{\frac{1}{p}}x^i.
$$
$\pi$ is 1-1 since $x,y$ and the $t_i^{\frac{1}{p}}$ are algebraically independent over $k$. The order valuation  on $\overline k[[x]]$
induces a rank 1 valuation $\nu$ on the quotient field of $R$.
Let $v_1=(\frac{v}{u})^p-t_1$.

$R_1=R[\frac{v}{u}]_{(u,v_1)}$ is dominated by $\nu$.
From the expansion
$$
v_1=\sum_{i=1}^{\infty}t_{i+1}x^{ip},
$$
we inductively define
$$
v_{j+1}=\frac{v_j}{u^p}-t_{j+1}=\sum_{i=1}^{\infty}t_{i+j}x^{ip}
$$
 and
$$
R_{j+1}=R_i[\frac{v_j}{u^p}]_{(u,v_{j+1})}
$$
for $j\ge 1$.
The $R_j$ are dominated by $\nu$ for all $j$, so that $V=\cup_{j\ge 1}R_j$
is the valuation ring of $\nu$.   We have that the residue field of $V$ is $V/m_V=R_1/m_{R_1}=k(t_1^{\frac{1}{p}})$. This is a finite extension of $k$,
in contrast to the fact that the field of coefficients $L=k(\{t_i^{\frac{1}{p}}\mid i\in\NN\})$ of $\sigma(x)$
has infinite degree over $k$.
\end{proof}

An especially strange representation  of a rank 2 discrete valuation is
given by the example (\ref{eqC}) of a power series whose exponents have  unbounded denominators.

 Let $k$ be a field of characteristic $p>0$, and consider the series
\begin{equation}\label{strseries}
\sigma=\sum_{i=1}^{\infty}x^{1-\frac{1}{p^i}}
\end{equation}
of (\ref{eqC}). $\sigma$ is algebraic over $k(x)$, with irreducible relation $\sigma^p-x^{p-1}\sigma-x^{p-1}=0$.

Consider the two dimensional regular local ring
$R_0=k[x,y]_{(x,y)}$. $x$ and $y$ are regular parameters in $R_0$. Let $y=\sigma(x)$. We see from (\ref{strseries}) that $y$ does not have a
fractional power series representation in terms of $x$. However, by expanding $x$ in terms of $y$, we have an expansion
\begin{equation}\label{eqseries}
x=y^{\frac{p}{p-1}}(1+y)^{-\frac{1}{p-1}}
\end{equation}
which represents $x$ as a  fractional power series in $y$ with bounded denominators.

Let $g=y^p-x^{p-1}y-x^{p-1}\in R_0$.  $g=0$
has a singularity of order $p-1$ in $R$. Let
$$
R_1=R[\frac{x}{y},y]_{(\frac{x}{y},y)}.
$$
$x_1=\frac{x}{y}$ and $y$ are regular parameters in $R_1$. $g=y^{p-1}g_1$, where
$$
g_1=y-x_1^{p-1}y-x_1^{p-1}
$$
is a strict transform of $g$ in $R_1$. $g_1=0$ is nonsingular.
From the equation $g_1=0$ we deduce that
$$
\begin{array}{lll}
y&=&x_1^{p-1}(1-x_1^{p-1})^{-1}\\
&=&x_1^{p-1}(1+x_1^{p-1}+x_1^{2(p-1)}+\cdots)\\
&=&\sum_{i=1}^{\infty}x_1^{i(p-1)},
\end{array}
$$
obtaining a standard power series expansion of $y$ in terms of $x$.

We obtain a fractional power series of $x_1$ in terms of $y$ with bounded denominators either from the equation
$g_1=0$, or by substitution in (\ref{eqseries}).

\par\vskip.5truecm\noindent
Steven Dale Cutkosky,  \hfill Olga Kashcheyeva,\par\noindent Dept.
Math.\hfill Dept. Math.\par\noindent University of Missouri,\hfill
University of Illinois, Chicago\par\noindent Columbia, MO 65211
USA\hfill Chicago, IL 60607 \par\noindent
cutkoskys@missouri.edu\hfill olga@math.uic.edu\par\noindent

\end{document}